\documentclass{amsart}

\newtheorem{theorem}{Theorem}[section]
\newtheorem{lemma}[theorem]{Lemma}
\newtheorem{proposition}[theorem]{Proposition}
\newtheorem{corollary}[theorem]{Corollary}

\theoremstyle{definition}
\newtheorem{definition}[theorem]{Definition}
\newtheorem{example}[theorem]{Example}

\newtheorem{conjecture}[theorem]{Conjecture}

\theoremstyle{remark}
\newtheorem{remark}[theorem]{Remark}

\numberwithin{equation}{section}

\begin{document}

\title{Generalized modular equations and
the CM values of Hauptmoduln}

\author{Kazuki Tomiyama}
\address{Department of Mathematics Faculty of Science and Engineering Waseda University
 3-4-1 Okubo, Shinjuku-ku, Tokyo 169-8555 JAPAN}
\email{tomiyama@suou.waseda.jp}


\date{September 25, 2026 and, in revised form, .}


\begin{abstract}
Monstrous moonshine relates an outstanding representation of the Monster finite sporadic simple group to modular functions with distinguished properties, which are called Hauptmoduln.  
Chen-Yui~\cite{Chen-Yui} showed that the CM values of Hauptmoduln which appear in monstrous moonshine (but not all) are algebraic integers, which is similar to the singular moduli of the $j$-function. In this paper, we generalize this result to Hauptmoduln whose $q$-coefficients are cyclotomic integers. A main idea for our proof is the use of generalized modular equations for Hauptmoduln, which was introduced by Cummins-Gannon~\cite{Cummins-Gannon} in the study of monstrous moonshine.
As an application, we show that if a formal $q$-series satisfies the special combinatoric property called complete replicability, its CM values are algebraic integers.

\end{abstract}

\maketitle

\tableofcontents

\section{Introduction}
Throughout the present paper, $z$ denotes a variable in the upper half plane $\mathbb{H}$, and we write $q = e^{2\pi i z}$.

\subsection{Hauptmoduln and its CM values}
The elliptic modular $j$-function is a modular function for the modular group $SL_2(\mathbb{Z})$, and it plays a crucial role in various areas of mathematics, especially in number theory :
\[
  j(z) := 1728 \frac{g_2(z)^3}{g_2(z)^3-27g_3(z)^2}   
          = q^{-1} + 744 + 196884q +21493760q^2 + 864299970q^3 +\cdots ,
\]
where $g_2(z)= 60\cdot \sum_{(m,n)\neq (0,0)} (m+nz)^{-4},~ g_3(z)= 140\cdot \sum_{(m,n)\neq (0,0)} (m+nz)^{-6}$.
As the compactification of the Riemann surface $SL_2(\mathbb{Z})\backslash\mathbb{H}$ is of genus-zero (i.e. the quotient space $SL_2(\mathbb{Z})\backslash\mathbb{H}$ is isomorphic to the Riemann sphere), 
its modular function field is generated over $\mathbb{C}$ by one modular function, and indeed the $j$-function is a generator 
of this field.
In other words,
any modular function for $SL_2(\mathbb{Z})$ can be written as a rational function of $j$.\\

By generalizing this property of the $j$-function, we obtain the concept of
\emph{Hauptmoduln (principal moduli)}.
Let $\Gamma\subset SL_2(\mathbb{R})$ be a  Fuchsian group of the first kind,
and we assume that $\Gamma$ contains the translation $z\mapsto z+1$ so that modular functions for $\Gamma$ have the \emph{$q$-expansion} of the form
\[
\sum_{m\gg \infty}^\infty a_m q^m .
\]
We denote by $X_\Gamma$ the compactification of the Riemann surface $\Gamma\backslash\mathbb{H}$.
This discrete group $\Gamma$ is said to be \emph{genus-zero} if $X_\Gamma$ has genus-zero as a compact Riemann surface.
If $\Gamma$ has genus-zero, there exists a modular function which is a generator of its modular function field.
Although the choice of generators is not unique, we can take it uniquely so that it has a $q$-expansion of the form 
\[
q^{-1} + \sum_{n=1}^\infty a_n q^n .
\]
We call such distinguished modular function 
the (normalized) \emph{Hauptmodul} for the genus-zero group $\Gamma$.
For example, $J(z):= j(z)-744 = q^{-1}+O(q)$ is the Hauptmodul for $SL_2(\mathbb{Z})$.\\

It is well-known that the
special values of the $j$-function at imaginary quadratic numbers (\emph{CM points}) are algebraic integers (cf.~\cite{Cox},\cite{Lang elliptic}). 
These CM values are called \emph{singular moduli}, and have rich arithmetic properties. 
To give an important example, the singular moduli generate the ring class fields over certain imaginary quadratic number fields. \\

The arithmetic properties of the CM values of Hauptmoduln were first studied in \cite{Chen-Yui},
and subsequently have been explored in several papers (cf.~\cite{principal moduli},\cite{Jorgenson}, \cite{Koo Dong Shin}).
The aim of the present paper is to prove the algebraic integrality of the CM values of a certain class of Hauptmoduln by utilizing the notion of \emph{generalized modular equations}, which was introduced by Cummins-Gannon~\cite{Cummins-Gannon} in the context of \emph{monstrous moonshine}.

\subsection{Monstrous moonshine and previous research}
\emph{Moonshine} is a surprising relation between certain modular objects (such as modular functions, modular forms and mock modular forms) and an outstanding representations of certain finite groups (cf.~\cite{Moonshine}).
Hautmoduln play a crucial role in moonshine of the \emph{Monster group}.\\

According to the classification of finite simple groups,
there are twenty six sporadic finite simple groups (cf.~\cite{ATLAS}).
The largest one has about $10^{54}$ elements, and is called the \emph{Monster group}.
\emph{Monstrous moonshine} (cf.~\cite{Borcherds},\cite{Conway-Norton}) 
asserts that there exists a graded infinite dimensional $\mathbb{M}$-module $V^\natural= \oplus_{n=-1}^\infty V_n^\natural$,
today called the \emph{moonshine vertex operator algebra} (cf.~\cite{F-L-M}), such that for each element $g\in\mathbb{M}$, 
the graded-trace function
\[
T_g(z) := \sum_{n=-1}^\infty tr(g|V_n^\natural) q^n 
\]
is a Hauptmodul for some genus-zero subgroup of $SL_2(\mathbb{R})$.
These $T_g$ are called the \emph{McKay-Thompson series} in monstrous moonshine, and the one corresponding to the identity element $e$ of the Monster group is $J(z)$, a Hauptmodul for $SL_2(\mathbb{Z})$. That is,
\begin{align*}
  T_e(z) &= \sum _{n=-1}^\infty \text{dim}(V^\natural_n) q^n    
          = q^{-1}  + 196884q +21493760q^2 + 864299970q^3 +\cdots  \\
          &= j(z) -744 \; =\;J(z).\\
\end{align*}
By definition, McKay-Thompson series $T_g$ depend only on the conjugacy class of $g\in\mathbb{M}$.
The Monster group has 194 conjugacy classes, and it is known that 
In addition, all the McKay-Thompson series possess a remarkable combinatoric property, called \emph{complete replicability},
which
we will study in Section \ref{replicable}.\\

Chen-Yui~\cite{Chen-Yui} studied the CM values of some McKay-Thompson series in monstrous moonshine, and showed that these CM values
have arithmetic properties which are similar to the singular moduli of the $j$-function.
We will give an overview of these results.\\

Let $\Gamma_g\subset SL_2(\mathbb{R})$ be a genus-zero 
subgroup corresponding to an element $g\in\mathbb{M}$ in monstrous moonshine.
It is known that $\Gamma_g$ contains a congruence subgroup
\[
\Gamma_0(N) =  \biggl \{\begin{pmatrix}
  a & b\\
  c & d
\end{pmatrix}\in SL_2(\mathbb{Z}) \:|\: c\equiv0 \mod N \biggr \}
\]
for some $N\geq1$, and $N$ is divisible by the order of $g\in\mathbb{M}$.
A McKay-Thompson series $T_g(z)$
is called \emph{fundamental} if the smallest $N$ is exactly the order of $g\in\mathbb{M}$.

\begin{theorem}[{\cite[Theorem~3.4]{Chen-Yui}}]
  Let $T_g$ be a fundamental McKay-Thompson series, and suppose that the invariance group $\Gamma_g$ contains $\Gamma_0(N)$. Let $\tau\in\mathbb{H}$ be a root of $az^2+bz+c=0~(a,b,c\in\mathbb{Z}, (a,N)=1)$.
  Then the CM value $T_g(\tau)$ is an algebraic integer.
\end{theorem}

One of the key points of their proof is constructing the \emph{modular equations} (see Section 2), by using the fundamental property of the McKay-Thompson series in question.
After the work of \cite{Chen-Yui}, 
generalizations of this result to Hauptmoduln with rational $q$-coefficients~(cf.~\cite{principal moduli},\cite{Koo Dong Shin}),
or to special values at any elliptic fixed points (cf.~\cite{Jorgenson}) have been studied,
but by other methods than that of \cite{Chen-Yui}.
In the present paper, we will consider Hauptmoduln whose coefficients are not necessarily rational numbers, extending method in \cite{Chen-Yui}.\\

\subsection{Main results}

In the present paper, we study the CM values of Hauptmoduln with cyclotomic integer $q$-coefficients.  
The following are the main results of this paper :

\begin{theorem}[Theorem~\ref{main thm}]
   Let $\Gamma$ be a discrete subgroup of $SL_2(\mathbb{R})$ of genus-zero,
   and suppose that $\Gamma$ contains a congruence subgroup $\Gamma_0(N)$ for some $N\geq1$.
   We denote $h$ by the Hauptmodul for genus-zero group $\Gamma$,
   and suppose that all the $q$-coefficients of $h$ are contained in the ring of integers of some cyclotomoic field. 
  If $\tau\in\mathbb{H}$ be a root of $az^2+bz+c=0~(a,b,c\in\mathbb{Z}, (a,N)=1)$, then the CM value $h(\tau)$ is an algebraic integer.
\end{theorem}

For our proof, 
we will utilize the notion of \emph{generalized modular equations},
which was introduced in \cite{Cummins-Gannon} for filling \emph{'conceptual gaps'} of monstrous moonshine (cf.~\cite{Gannon_Hauptmodul}).
Roughly speaking, generalized modular equations are modular equations twisted by the Galois actions of a cyclotomic field, and one can characterize Hauptmoduln among $q$-series in an algebraic sense.\\

Due to the Galois action on modular equations, we need to restrict ourselves to Hauptmodul with
cyclotomic integer coefficients,
however, this assumption is not unnatural in the sense that 
Hauptmoduln for 
all the genus-zero discrete subgroups of \emph{moonshine-type}~(see Definition \ref{moonshine-type}), especially all the McKay-Thompson series in monstrous moonshine,
satisfy the assumptions of the theorem.\\

The paper is outlined as follows.
In Section 2, we will overview the definitions and some facts on (generalized) modular equations. 
In Section 3, we will give a proof of the main result of the present paper. 
Section 4 is devoted to the study of completely replicable functions, as an application of the main result.

\section*{Acknowledgement}
This is part of the author's master's thesis, written under the supervision of Professor Hiro-aki Narita at Waseda University.
The author wishes to express his deepest gratitude to his supervisor for his dedicated guidance.
The author also gratefully acknowledges the many helpful suggestions of Dr.~Shuji Horinaga~(Institute for Fundamental Mathematics in NTT) and Dr.~Shotaro Kimura~(Waseda University).
Especially, the ideas of our proof of Lemma.\ref{mugennkouka}, which is essential for this paper, is given by S.~Horinaga.
The author's profound gratitude is also due to Professor John~F.R.~Duncan~(Academia Sinica) for his kindly guiding the author to various aspects of moonshine.
He also corrected many English errors in this paper.
The author thanks Professor Claudia Alfes (Bielefeld University) for her careful reading of the manuscript and for her helpful comments.
Professor Dmitry Kozlov (University of Bremen) and Professor Arne Meurman (Lund University) kindly sent the author a copy of Kozlov's master's thesis \cite{Kozlov}.
Finally, the author would like to thank his family for their support and understanding of his path of researching in mathematics.

\section{Modular equations}

\subsection{Modular equations for the $j$-function}
The modular equations for the $j$-function are algebraic relations between $j(z)$ and $j(nz)$ for positive integers $n$, and
these relations deduce some arithmetic properties of singular moduli.
Here, we will review the classical theory of modular equations (cf.~\cite{Lang elliptic}).\\

For a positive integer $n>1$, we define a set of matrices with integer entries
\[
\Omega(n) := \Bigl \{ \begin{pmatrix}
   a & b \\
   0 & d\end{pmatrix}   \in\text{Mat}_2(\mathbb{Z})         \; |\; \text{gcd}(a,b,d)=1, ad=n, 0\leq b < d\Bigr \}.
\]
We now introduce the notion of \emph{modular polynomials} for formal $q$-series.
\begin{definition}
Let $f(q)=q^{-1}+\sum_{m=1}^\infty a_mq^m\;(a_m\in\mathbb{C})$ be a formal $q$-series.
  The \emph{modular polynomial} of order $n$ for the formal series $f(q)$ is a monic polynomial with two variables $F_n^f(X,Y)\in \mathbb{C}[X,Y]$ which satisfies the following :
 \begin{itemize}
 \item $F_n^h(X,Y)=F_n^h(Y,X)$.
 \item The degree of $F_n^f(X,Y)$ is $\psi(n) := n \prod_{p|n} (1+\frac{1}{p})$ in $X$ and $Y$. Here $p$ runs primes dividing $n$.
 \item For any $\omega \in\Omega(n)$, the polynomial $F_n^f(X,Y)$ satisfies the following equations
 \[
 F_n^f\bigl(f(q),(f\circ \omega)(q) \bigr) =0.
 \]
\end{itemize}
  
\end{definition}
The above equation is called the \emph{modular equation} of order $n$ for the formal series $f(q)$.
For a formal series $f(q)$ and a positive integer $n>1$, the modular polynomial (or modular equation) $F_n^f(X,Y)$ does not always exist however,
for the $j$-function, we can construct modular equations of any order $n>1$.\\

For a positive integer $n>1$, we define a polynomial of $Y$ by
\[
F_n^j(Y) := \prod_{\omega\in\Omega(n)} (Y-j\circ\omega).
\]
By definition, the coefficients of $F_n^j(Y)$ are symmetric polynomials of $\{j\circ\omega \}_{\Omega(n)}$,
and it turns out that these are modular functions for $SL_2(\mathbb{Z})$. 
Hence the Hauptmodul property of $j$ and some considerations on their poles
asserts that 
all the coefficients of $F_n^j(Y)$ are polynomials of $j$. 
Since $\mathbb{C}[X,Y] \simeq \mathbb{C}[j,Y]$, there exists a unique polynomial with two variables $F_n^j(X,Y)$ satisfying
\[
F_n^j(j,Y) = F_n^j(Y),
\]
which satisfies the conditions defining the modular polynomial of order $n$.\\

Thus, the $j$-function satisfies the modular equations of order $n$ for any $n>1$,
and in particular if we take $\omega= \bigl(
\begin{smallmatrix}
   n & 0 \\
   0 & 1
\end{smallmatrix}
\bigl)$, the modular equation gives an algebraic relation between $j(z)$ and $j(nz)$.
Furthermore, it can be shown that the modular polynomials $F^j_n(X,Y)$ is monic and have rational integer coefficients, 
and for a CM point $\tau\in \mathbb{H}$, the special value $j(\tau)$ is a root of $F_n^j(X,X) \in \mathbb{Z}[X]$, which deduce that this singular modulus $j(\tau)$ is an algebraic integer.\\

Finally, we remark that the property of satisfying many modular equations essentially characterizes the $j$-function among $q$-series.
\begin{theorem} [\cite{Kozlov}]
  Let $f(q)=q^{-1}+\sum_{m=1}^\infty a_mq^m\;(a_m\in\mathbb{C})$ be a formal $q$-series.
  The $q$-series $f(q)$ satisfies a modular equation of order $n$ for all $n>1$
  if and only if  $f(q)$ is either $f(q)=J(q)=j(q)-744$, $f(q)=q^{-1}$, or $f(q)=q^{-1}\pm q$.
\end{theorem}

\subsection{Generalized modular equations}
Monstrous moonshine relates the representation theory of the monster group $\mathbb{M}$ to distinguished modular functions, Hauptmoduln.
The theory of vertex operator algebra somewhat explain why the representation of $\mathbb{M}$ yields modular functions (c.f.~\cite{Zhu}) however,
it does not provide reasons why these modular functions have the genus-zero property.\\

In the proof of monstrous moonshine, Borcherds~\cite{Borcherds} utilizes the theory of  infinite dimensional Lie algebras to 
prove the McKay-Thompson series of $\mathbb{M}$ satisfy
certain algebraic recursion property, called \emph{replicability} (see Section \ref{replicable}).
Replicability imposes a strong restriction on the Fourier coefficients of the McKay-Thompson series,
which implies that
all the Fourier coefficients coincide with those of Hauptmoduln which are already known.
Thus, the background of the appearance of Hauptmoduln in monstrous moonshine remained unclear even after the complete proof by Borcherds,
which is called the "conceptual gap" (c.f.~\cite{Gannon_Hauptmodul}).\\

The problem is that however the genus-zero property of Hauptmoduln arises from the topology of Riemann surfaces, replicability is an algebraic or combinatorial property.
Cummins-Gannon~\cite{Cummins-Gannon} were motivated to fill this gap, 
and introduced the notion of \emph{generalized modular equation}, in order to interpret the genus-zero property
into algebraic settings.
Here, we will overview this theory.\\

Fix a positive integer $N\geq 1$, and let us denote by $\zeta_N$ a primitive root of unity of order $N$.
For an integer $n>1$ which is coprime to $N$, we will
denote by $*n$
an element of Gal$(\mathbb{Q}(\zeta_N)/\mathbb{Q})$ such that 
\[
\zeta_N *n = \zeta_N^n.
\]
Also, for a formal series $f(q)$, we denote by $f*n$
the $q$-series obtained from $h$ by applying $*n$ to each coefficient.
We now define the generalized modular polynomials for formal $q$-series.

\begin{definition}\label{moonshine-type}
  Let $f(q)= q^{-1}+ \sum_{m=1}^\infty a_mq^m$ be a formal $q$-series with $\mathbb{Q}(\zeta_N)$-coefficients, and $n>1$ be an integer coprime to $N$. 
The \emph{generalized modular polynomial} of order $n$ for the formal series $f(q)$ is a monic polynomial with two variables $F_n^f(X,Y)\in \mathbb{C}[X,Y]$ which satisfies the following :
 \begin{itemize}
 \item $F_n^h(X,Y)=(F_n^h*n)(Y,X)$.
 \item The degree of $F_n^f(X,Y)$ is $\psi(n) := n \prod_{p|n} (1+\frac{1}{p})$ in $X$ and $Y$. Here $p$ runs primes dividing $n$.
 \item For any $\omega \in\Omega(n)$, the polynomial $F_n^f(X,Y)$ satisfies the \emph{generalized modular equations} of order $n$
 \[
 F_n^f\bigl((f*n)(q),(f\circ \omega)(q) \bigr) =0.
 \]
\end{itemize}
\end{definition}

In the work of \cite{Cummins-Gannon}, it is shown that the property of satisfying many generalized modular equations essentially characterize some natural class of Hauptmoduln, called \emph{moonshine-type}.  
\begin{definition} \label{moonshine-type}
  A discrete subgroup $\Gamma\subset SL_2(\mathbb{R})$ is said to be \emph{moonshine-type} if the following conditions are satisfied :
\begin{quote}
 \begin{itemize}
  \item $\Gamma$ contains $\Gamma_0(N)$ for some $N>0$.
  \item $\begin{pmatrix}
   1 & t \\
   0 & 1\end{pmatrix} \in \Gamma \; \Rightarrow \; t\in \mathbb{Z}$.
 \end{itemize}
\end{quote}
\end{definition}

\begin{theorem}[{ \cite[Theorem.1.3,1.4]{Cummins-Gannon}}]
 Let $f(q)=q^{-1}+\sum_{m=1}^\infty a_mq^m$ be a formal $q$-series with $\mathbb{Q}(\zeta_N)$-coefficients.
  The $q$-series $f(q)$ satisfies a generalized modular equation of order $n$ for all $n\equiv 1 \mod N$
  if and only if it is either of the form $f(q)=q^{-1}+a_1q$ or is a Hauptmodul for some genus-zero moonshine-type group.
 
\end{theorem}

Thus, the generalized modular equations gives an algebraic interpretation of the notion of Hauptmoduln.
In particular, Hauptmoduln in this theorem satisfy generalized modular equations
of order $n$ for any $n\equiv 1\mod{N}$.
Since we do not need so many modular equations for our purpose,
we will consider slightly more general class of Hauptmoduln in the next section.\\

Also, the relation between modular equations and replicability has been studied (cf.~\cite{Carnahan 1},\cite{Carnahan 2},\cite{Cummins-Gannon},\cite{Kozlov}),
and we will come back to this topic in Section~\ref{replicable} of the present paper.

\section{Proof of the main results}
In this section, we will give a proof of the main result of the present paper (Thoerem \ref{main thm}), which asserts that the CM values of Hauptmoduln with cyclotomic integer $q$-coefficients are algebraic integers.
The proof is similar to the case of the elliptic modular $j$-function however,
there are some obstacles because we have to consider the generalized modular equations.

\subsection{Construction of generalized modular equations}
First, we construct the generalized modular equations for our case, following the work of \cite{Cummins-Gannon}.\\

Let $\Gamma \subset SL_2(\mathbb{R})$ be a discrete subgroup of genus-zero, and suppose that $\Gamma$ contains a congruence subgroup $\Gamma_0(N)$ for some $N\geq 1$.
Let 
\[
h(z) = q^{-1} + \sum_{m=1}^\infty a_m q^m 
\]
be the Hauptmodul for the genus-zero group $\Gamma$.
Also, we suppose that all the $q$-coefficients of the Hauptmodul $h$ for $\Gamma$ are contained in a cyclotomic field $\mathbb{Q}(\zeta_k)$ for some $k>1$ :
\[
a_m \in \mathbb{Q}(\zeta_k) \;\;(\forall m\geq 1).
\]

We remark that without loss of generality, we can assume that all the $q$-coefficients of $h$ 
are contained in the cyclotomic field $\mathbb{Q}(\zeta_N)$.
This is because even if $a_m\in\mathbb{Q}(\zeta_k) \not\subset \mathbb{Q}(\zeta_N)$, 
we have $\Gamma \supset \Gamma_0(N) \supset \Gamma_0(kN)$ and $a_m \in\mathbb{Q}(\zeta_k)\subset \mathbb{Q}(\zeta_{kN})$,
hence by replacing $N$ by $kN$, we have
\[
\Gamma\supset\Gamma_0(N),\;\; a_m\in\mathbb{Q}(\zeta_N).
\]

Now, for a positive integer $n>1$ which is coprime to $N$, we consider a polynomial with one variable
\[
 F_n^h(Y)  := \prod _{\omega\in\Omega(n)} (Y-h\circ\omega).
\]
Each coefficient of $F_n^h(Y)$
is a symmetric polynomial of $\{ h\circ\omega \}_{\omega\in \Omega(n)}$,
and it turns out by the following proposition that these are all polynomials of 
$h*n$.

\begin{proposition}[{\cite[Proposition.6.16]{Cummins-Gannon}}]\label{symmetric polynomial}\:
  Let $f$ be a Hauptmodul for a genus-zero group $\Gamma\supset \Gamma_0(N)$, and $n$ be a positive integer coprime to $N$. 
  We suppose that the Hauptmodul $h$ has $\mathbb{Q}(\zeta_N)$-coeffients, and denote by $H_f$ the field generated by the $q$-coefficients of $f$ over $\mathbb{Q}$. 
  Then any symmetric polynomial of $\{f\circ\omega\} _{\omega\in\Omega(n)}$ is an element of $H_f[f*n]$.
\end{proposition}

Hence we have
\[
F_n^h(Y) \in \mathbb{Q}(\zeta_N)[h*n,Y],
\]
and since $\mathbb{Q}(\zeta_N)[h*n,Y] \simeq \mathbb{Q}(\zeta_N)[X,Y]$, we get a
unique polynomial with two variables $F_n^h(X,Y)\in\mathbb{Q}(\zeta_N)[X,Y]$ such that
\[
F_n^h(h*n,Y) = F_n^h(Y).
\]
By its construction, $F_n^h(X,Y)$ satisfies the generalized modular equations of order $n$
\[
   F_n^h\bigl((h*n)(q), (h\circ\omega)(q)\bigr) = 0 \;\;\;\; (\forall \omega\in\Omega(n)).
\]
Indeed, it turns out that these polynomials are the generalized modular polynomials for the Hauptmodul $h$ (cf.~\cite[Proposition 6.18]{Cummins-Gannon}). \\

\subsection{Coefficients of the generalized modular polynomials}
Here, we compute coefficients of the diagonal restriction of generalized modular polynomials.\\

First, we take appropriate orders of generalized modular polynomials so that the Galois action acts to the $q$-coefficients trivially.
Let $\mathfrak{F}_N$ denote the field of modular functions for the principal congruence group $\Gamma(N)$ whose $q$-coefficients are contained in the cyclotomic field $\mathbb{Q}(\zeta_N)$.
Then it is known that for $n$ coprime to $N$ the Galois action $*n^2$ gives an automorphism of the field $\mathfrak{F}_N$ (cf.~\cite{Shimura}).
Due to this fact and Hauptmodul property of $h$, we get the following lemma.

\begin{lemma}[{\cite[Proposition.6.14]{Cummins-Gannon}}]\label{square invariance}
Let $n$ coprime to $N$, and $\Gamma$ be a genus-zero subgroup of $SL_2(\mathbb{R})$. Suppose that the Hauptmodul for $\Gamma$ has $q$-coefficients in 
$\mathbb{Q}(\zeta_N)$. 
Then we have
\[
h*n^2 =h.
\]
\end{lemma}

Therefore, by considering square orders, we get the following properties of the generalized modular equations which we have constructed.

\begin{proposition}\label{modular equation two variants}
 Let $n$ be a square integer coprime to $N$, and $\Gamma$ be a genus-zero subgroup of $SL_2(\mathbb{R})$. Suppose that the Hauptmodul for $\Gamma$ has $q$-coefficients in 
$\mathbb{Q}(\zeta_N)$.
Then we have
  \[
  F_n^h(h*n,Y) = F_n^h(h,Y) = F_n^h(Y) \in \mathbb{Q}(\zeta_N)[h,Y].
  \]
\end{proposition}

Now, we prepare the next lemma on the cyclotomic polynomials
$\Phi_n(X)=\prod_{1\leq k\leq n,(k,n)=1}(X-\zeta_n^k)$ for our proof of the main results. For the proofs we refer the reader to \cite{Lang alg} (Chapter \rm{I}\hspace{-1pt}\rm{V},~Theorem 1.1).
\begin{lemma}\label{cyclotimic polynomials}
It holds
  \begin{equation*}
  \Phi_n(1)=
  \begin{cases}
    p & \text{if $n$ is a prime power $n=p^e$}, \\
    1                 & \text{otherwise.} 
  \end{cases}
\end{equation*}
\end{lemma}

From now on, we suppose that the $q$-coefficients of Hauptmoduln are contained in  the ring of integers $\mathbb{Z}[\zeta_N]$,
consider the diagonal restriction $F_n^h(h,h)$ for square order $n$.
We note that its coefficients are generally contained in $\mathbb{Q}(\zeta_N)$, not $\mathbb{Z}[\zeta_N]$, because the field $H_h$ in Proposition.\ref{symmetric polynomial} is $\mathbb{Q}(\zeta_N)$.

\begin{proposition}\label{coefficients modular equations}
Let $n$ be a square integer coprime to $N$, and $h$ be a Hauptmodul whose Fourier coefficients are contained in $\mathbb{Z}[\zeta_N]$.
Then we have 
\[
F_n^h(h,h) \in \mathbb{Z}[\zeta_N][h].
\]
Furthermore, its leading coefficient is\\
(a) a unit of the integer ring $\mathbb{Z}[\zeta_N]$, if $n$ has two or more prime divisors.\\
(b) represented as $p\times u$ for some unit $u\in \mathbb{Z}[\zeta_N]$, if $n$ is a power of prime $p$.

\end{proposition}

\begin{proof}
By Proposition.\ref{modular equation two variants},
we have 
\[
F_n^h(h,h) \in \mathbb{Q}(\zeta_N)[h].
\]
Since $F_n^h(h,h)$ is a polynomial of the Hauptmodul $h$, $F_n^h(h,h)$ has $q$-expansion of the form 
\begin{align*}
  F_n^h(h(z),h(z)) &=    \prod_{\omega\in\Omega(n)}\bigl(h(z)-h(\omega z)\bigr)\\
          &=  \prod_{\bigl(\begin{smallmatrix}
            a & b \\
            0 & d 
          \end{smallmatrix}\bigr)\in\Omega(n)}
          \Bigl\{\bigl(q^{-1}+ \Sigma_{m=1}^\infty a_m q^m)- (\zeta_d^{-b}q^{-a/d}+\Sigma_{m=1}^\infty a_m \zeta_d^{b} q^{am/d}\bigr)\Bigr\},
          \end{align*}
and it follows that its $q$-coefficients are contained in $\mathbb{Z}[\zeta_N]$.
We shall write its $q$-expansion as
\[
F_n^h(h,h) = \sum_{m=-m_0}^{\infty} b_m q^m\;\;(b_m\in\mathbb{Z}[\zeta_N], m_0\in\mathbb{Z}_{>0}, b_{-m_0}\neq0).
\]

Now, since the Hauptmodul $h$ has $q$-expansion of the form $h(q)=q^{-1}+ O(q)$, it can be seen that the leading term of the polynomial $F_n^h(h,h)$ is $b_{-m_0}h(q)^{m_0}$.
If we put $h(q)^{m_0}=q^{-m_0}+\sum_{m>-m_0}c_mq^m~(c_m\in \mathbb{Z}[\zeta_N])$, we have  
\begin{align*}
 F_n(h,h)-b_{-m_0}h(q)^{m_0}  &= b_{-m_0}q^{-m_0}+b_{-m_0+1}q^{-m_0+1}+\cdots \\
&\;\;\;\;\;\;\;\; -b_{-m_0}q^{-m_0} - b_{-m_0}c_{-m_0+1}q^{-m_0+1}-\cdots \\
          &=  (b_{-m_0+1}-b_{-m_0}c_{-m_0+1})q^{-m_0+1}+\cdots.
\end{align*}

Therefore, we can see that
the second leading coefficient of $F_n^h(h,h)$ is $(b_{-m_0+1}-b_{-m_0}c_{-m_0+1})$.
In the same manner, it follows that all the coefficients of $F_n^h(h,h)$ are contained in the ring generated by the $q$-coefficients of $F_n(h,h)$ and $h(z)$, therefore we conclude that the coefficients of polynomial $F_n^h(h,h)$ are contained in $\mathbb{Z}[\zeta_N]$.\\

Now, we calculate the leading coefficients of $F_n^h(h,h)$.
Fix $\omega = \begin{pmatrix}
  a & b \\
  0 & d 
\end{pmatrix}\in\Omega(n)$,
and we write
\[
h-h\circ\omega = \bigl(q^{-1}+ \Sigma_{m=1}^\infty a_m q^m)- (\zeta_d^{-b}q^{-a/d}+\Sigma_{m=1}^\infty a_m \zeta_d^{b} q^{am/d}\bigr).
\]
Then, the leading term of $h-h\circ\omega$ with respect to $q$ is
\begin{equation*}
  \begin{cases}
    q^{-1} &  \text{ ($a<d$),}\\
    -\zeta_d^{-b}q^{-a/d}                 & \text{ ($a>d$),} \\
   (1-\zeta_d^{-b})q^{-1}       & \text{($a=d$).}
  \end{cases}
\end{equation*}

If $a=d$, since $ad=n,(a,b,d)=1$, we get $(a,b)=(d,b)=(\sqrt{n},b)=1$, therefore we note that $b\neq0$. Thus, we can write the leading term of $F_n(h,h)=\prod_{\omega\in\Omega(n)}(h-h\circ\omega)$ as follows :
\[
\prod_{1\leq b < d=\sqrt{n},(b,\sqrt{n})=1}(1-\zeta_{\sqrt{n}}^{-b})q^{-1} \cdot \prod_{ad=n,a>d , 0\leq b <d}q^{-1} \cdot \prod_{ad=n,a<d, 0\leq b <d}-\zeta_d^{-b}q^{-a/d} .
\]

By Lemma \ref{cyclotimic polynomials},  the most left term of  this product is $1$ (case (a)) or $p$ (case (b)).
We have seen that the leading coefficient of the $F_n^h(h,h)$ coincides with the leading coefficient of this product,
which completes the proof.

\end{proof}

\subsection{Proof of results}
In this section, we will give a proof of the main result of this paper (Theorem \ref{main thm}).
So far we have studied the coefficients of the generalized modular polynomials $F_n^h(X,X)$.
Here, we will prove that the CM value of the Hauptmodul $h$ is a solution of the modular equation $F_n^h(X,X)=0$. 
Since we know that $F_n^h(X,Y)$ satisfies modular equations 
\[
F_n^h(h(z),h(\omega z)) =0 \;\;\;(\forall \omega\in\Omega(n)),
\]
it is sufficient to find a matrix $\omega\in\Omega(n)$ such that $h(\omega\tau) = h(\tau)$ for a CM point $\tau$.
\begin{lemma}[{\cite[Lemma.2.2]{Chen-Yui}}]\label{matrix lemma}
Let $\rho=\bigl( \begin{smallmatrix}
  a & b\\
  c & d
\end{smallmatrix}\bigr) \in \text{M}_2(\mathbb{Z})$ be an integer matrix with 
gcd$(a,b,c,d)=1$, and $m$ denote the determinant of $\rho$. Let $N\geq 1$ be an integer. Suppose that $gcd(a,N)=1 \; \text{and} \; N|c$. 
Then there exist matrices $\gamma\in\Gamma_0(N)$ and $\omega\in\Omega(m)$ such that $\rho =\gamma\omega$.
\end{lemma}
By the above lemma, for a CM point $\tau\in\mathbb{H}$, if we get a such matrix $\rho \in\text{M}_2(\mathbb{Z})$ with determinant $n$ satisfying
$\rho\tau=\tau$,
then we get 
\[
h(\tau) = h(\rho\tau) = h(\gamma\omega\tau) = h(\omega\tau)
\]
since the Hauptmodul $h$ is $\Gamma_0(N)$-invariant.
Thus we only need to find a integer matrix $\rho$ with square determinant acts trivially to $\tau$.\\

We first consider the case that the CM point $\tau\in\mathbb{H}$ is an algebraic integer.
\begin{proposition} \label{algebraic integer points}
  Let $K$ be an imaginary quadratic field, $\mathcal{O}_K$ be the ring of integers of $K$, and $\tau_0\in \mathcal{O}_K\cap \mathbb{H}$.
  Let $h$ be a Hauptmodul for a genus-zero group $\Gamma\supset \Gamma_0(N)$. Suppose that the $q$-coefficients of $h$ are contained in some cyclotomic integer ring.
  Then the CM value $h(\tau_0)$ is an algebraic integer.
\end{proposition}

\begin{proof}
 Let us denote by $Q[a,b,c]$ a quadratic form $ax^2+bxy+cy^2$,
  and we consider a quadratic form $Q[\text{N}(\tau_0)N^{2},\text{Tr}(\tau_0)N,1]$,
  here N$(\tau_0)$ and Tr$(\tau_0)$ is the norm and trace of the algebraic integer $\tau_0$ respectively.
  Since the quadratic form $Q$ is primitive and positive definite, it represents infinitely many primes (cf.~\cite{Cox}).
  Thus we can take a prime $l$ not dividing $N$ so that there exist integers $c_l$, $d_l$, and 
  \[
   Q(c_l,d_l)= \text{N}(\tau_0)N^{2}c_l^2+\text{Tr}(\tau_0)Nc_ld_l+d_l^2= l.
  \]
  We now fix such a prime $l$.
  If we assume that $d_l^2$ and $N$ have a common divisor greater than $1$, it must be prime $l$, and this contradicts with $l\mathrel{\not |} N$.
  Therefore we have
  $(d_l^2,N)=1$.\\

  Next, we define a matrix with integer entries
  \[
  \rho _l
   :=\begin{pmatrix}
     c_lN\text{Tr}(\tau_0)+d_l & -c_lN\text{N}(\tau_0) \\
   c_lN & d_l
   \end{pmatrix}.
  \]
  By definition,
  it follows that $\rho_l\tau_0=\tau_0$, and $\rho_l$ has determinant $Q(c_l,d_l)=l$.
  Here, we consider
  \[
  \rho_l^2 =  \begin{pmatrix}
  (c_lN\text{Tr}(\tau_0)+d_l)^2-c_l^2N^{2}\text{N}(\tau_0) & -c_l^2N^{2}\text{Tr}(\tau_0)\text{N}(\tau_0)-2c_ld_lN\text{N}(\tau_0)\\
  c_l^2N^{2}\text{Tr}(\tau_0)+2c_ld_lN & -c_l^2N^{2}\text{N}(\tau_0)+d_l^2
\end{pmatrix},
  \]
and it is clear that $\rho_l^2\tau_0=\tau_0$ and det$(\rho_l^2)=l^2$.
Also, the lower left entry of $\rho_l^2$ is divisible by $N$, and the upper left entry is coprime to $N$, since $(d_l^2,N)=1$.
Furthermore, since det$(\rho_l^2)=l^2$ and $l$ is a prime, if the four entries of $\rho_l^2$ have a non-trivial
common divisor, it must be the prime number $l$.\\

Although it is not clear that whether this $\rho_l^2$ satisfy the conditions of Lemma \ref{matrix lemma}, in other words whether $\rho_l^2$ is primitive, 
  we now claim that we can take prime number $l$ so that the matrix $\rho_l^2$ is primitive :

\begin{lemma}\label{mugennkouka}
  Let $p$ be a prime number represented by the quadratic form $Q$ above,and write $Q(c_p,d_p)=p$.
  Then there exist infinitely many prime numbers $p$ such that $\rho_p^2$ is a primitive matrix.
\end{lemma}

\begin{proof}
 Remind that there exist infinitely many prime numbers that are represented by the quadratic form $Q[\text{N}(\tau_0)N^{2},\text{Tr}(\tau_0)N,1]$.
  Assume that the matrix $\rho_p^2$ is not primitive i.e. its entries have a common divisor $p$, hence $\rho_p^2\in \text{M}_2(p\mathbb{Z})$.
  Then, as an element of $\text{M}_2(\mathbb{F}_l)$, we have $\rho_l^2=0$ and its trace is $0$ in $\mathbb{F}_p$ :
  \[
  \text{tr}(\rho_p) = c_p N \text{Tr}(\tau_0)+2d_p \equiv 0 \mod p.
  \]
 Since $Q(c_p,d_p)= \text{N}(\tau_0)N^{2}c_p^2+\text{Tr}(\tau_0)Nc_pd_p+d_p^2=p$, at least one of $c_p$ and $d_p$ is not divided by $p$.
 Now, we can assume that $p$ divides neither $2$, nor $N$, nor $\text{Tr}(\tau_0)$ because only finitely many primes divide these numbers.
Then, it follows that neither $c_p$ nor $d_p$ is divisible by $p$,
 and we get an equality
 \[
  d_p \equiv (-2)^{-1}c_pN\text{Tr}(\tau_0) \mod l.
  \]
Consequently, by $Q(c_p,d_p)= p\equiv 0 \mod p$, we get the following :
\begin{align*}
0 &\equiv  \text{N}(\tau_0)N^{2}c_p^2+\text{Tr}(\tau_0)Nc_pd_p+d_p^2  \mod p\\
    &\equiv  \text{N}(\tau_0)N^{2}c_p^2+\text{Tr}(\tau_0)Nc_p(-2)^{-1}c_pN\text{Tr}(\tau_0) +(-2)^{-2}N^2c_p^2\text{Tr}(\tau_0) \mod p\\
    &\equiv  c_p^2N^2(\text{N}(\tau_0)+(-2)^{-1}\text{Tr}(\tau_0)+(-2)^{-2}\text{Tr}(\tau_0)) \mod{p}\\
    &\equiv  c_p^2N^2(-2)^{-2}(4\text{N}(\tau_0)-\text{Tr}(\tau_0)) \mod{p}.
    \end{align*}
  Remind that we assume that $p$ divides none of $c_p$, $N$, and $2$. Therefore, we conclude that if the matrix $\rho_p^2$ is not primitive for a prime $p$, 
  then, with finitely many exceptions, the following equality
  \[
  4\text{Tr}(\tau_0)-\text{Nr}(\tau_0)^2 \equiv 0\mod p
  \]
  holds.
  Since only finitely many primes divide $ 4\text{Tr}(\tau_0)-\text{Nr}(\tau_0)^2$,
  it turns out that there are only finitely many primes $p$ such that the matrix $\rho_p^2$ is not primitive, which proves the lemma.
\end{proof}

By this lemma, there exist infinitely many primes $l$ such that $\rho_l^2$ satisfies the conditions of Lemma.\ref{matrix lemma}.
Hence, fix such a prime number $l$, and we can write
\[
\rho_l^2 = \gamma\omega 
\]
for some $\gamma\in \Gamma_0(N)$ and $ \omega \in \Omega(l^2)$.
Now we have $h(\tau_0)= h(\rho_l^2\tau_0)=h(\gamma\omega\tau_0)=h(\omega\tau_0)$,
this yields that the CM value $h(\tau_0)$ is a root of $F_{l^2}^h(X,X)=0$ :
\[
F_{l^2}^h(h(\tau_0),h(\tau_0))= F_{l^2}^h(h(\tau_0),h\circ\omega(\tau_0)) = 0.
\]
Remind that the polynomial $F_{l^2}^h(X,X)\in \mathbb{Z}[\zeta_N][X]$ has the leading coefficient which is a unit in $\mathbb{Z}[\zeta_N]$ except for a $l$-factor (Proposition \ref{coefficients modular equations}(b)).
Therefore it follows that $l\cdot h(\tau_0)$ is an algebraic integer. 
  In other words, let us denote by $\mathbb{O}$ the set of all algebraic integers, we get
  \[
  h(\tau_0)\in \frac{1}{l}\mathbb{O}.
  \]
  By Lemma \ref{mugennkouka}, we can take another prime $p$, and apply this argument again to obtain $ h(\tau_0)\in \frac{1}{p}\mathbb{O}$.
 Therefore we conclude that 
\[
h(\tau_0)\in \frac{1}{l}\mathbb{O}\cap \frac{1}{p}\mathbb{O} = \mathbb{O},
\]
i.e. $h(\tau_0)$ is an algebraic integer.
\end{proof}

Now, we prove the main results of this paper.

\begin{theorem}\label{main thm}
   Let $\Gamma$ be a discrete subgroup of $SL_2(\mathbb{R})$ of genus-zero,
   and suppose that $\Gamma$ contains a congruence subgroup $\Gamma_0(N)$ for some $N\geq1$.
   Let us denote by $h$ the Hauptmodul for the genus-zero group $\Gamma$,
   and suppose that all the $q$-coefficients of $h$ are contained in the ring of integers of some cyclotomoic field. 
  Let $\tau\in\mathbb{H}$ be a root of $az^2+bz+c=0~(a,b,c\in\mathbb{Z}, (a,N)=1)$, then the CM value $h(\tau)$ is an algebraic integer.
\end{theorem}

\begin{proof}
  Suppose $a\neq 1$.
  By multiplying integers to both sides of $a\tau^2+b\tau+c=0$ if needed, we can assume that $a$ is a square which has two or more prime divisors, and is coprime to $N$.
Now, $\tau$ is not necessarily an algebraic integer, but $a\tau =:\tau_0$ is.
For this CM point $\tau_0$,
we can write
\[
\tau = \omega_0\tau_0,\;\; \omega_0= \begin{pmatrix}
   1 & 0 \\
   0 & a\end{pmatrix}\in\Omega(a) .
   \]
Thus, we get
\[
F_a^h(h(\tau_0),h(\tau))=F_a^h(h(\tau_0),(h\circ\omega_0)(\tau_0)) =0,
\]
and by the above assumption of the integer $a$, it follows that $(h\circ\omega_0)(\tau_0)(=h(\tau))$ is integral over $\mathbb{Z}[\zeta_N][h(\tau_0)]$.
By Proposition \ref{algebraic integer points}, we can see that $h(\tau_0)$ is an algebraic integer,
and therefore, we conclude that $h(\tau)=h(\omega_0\tau_0)$ is integral over $\mathbb{Z}$, which proves the theorem.
\end{proof}

It is known that there exist exactly 6486 moonshine-type (see Definition.\ref{moonshine-type}) discrete subgroups of $SL_2(\mathbb{R})$ which have genus-zero, and that their Hauptmoduln have 
cyclotomic integer $q$-coefficients (cf.~\cite{Cummins},\cite{Gannon_Hauptmodul}).
We also note that 171 of these are the McKay-Thompson series in monstrous moonshine.
Hence we have the following corollary :

\begin{corollary} \label{moonshine-type_cor}
  Let $\Gamma~(\supset \Gamma_0(N))$ be a genus-zero moonshine-type discrete subgroup of $SL_2(\mathbb{R})$, and $h$ be the Hauptmodul for $\Gamma$. Let $\tau$ as the above theorem.
  Then the CM value $h(\tau)$ is an algebraic integer.
  In particular, for any McKay-Thompson series in monstrous moonshine $T_g$, its CM value $T_g(\tau)$ is an algebraic integer.
\end{corollary}

\section{CM values of replicable functions}\label{replicable}
In this section, we study the CM values of completely replicable functions, as an application of the proof of the main result.
For more detail of (complete) replicable functions, we refer the reader to \cite{More on}, \cite{replicable introduction}.
The concept of replicability relates to certain infinite product expression (cf.~\cite{Carnahan 1},\cite{Carnahan 2}, \cite{multiplicative symmmetry}), and to the Schwarzian derivatives (cf.~\cite{Schwarz}).

\subsection{Definition of replicability}
All the McKay-Thompson series in monstrous moonshine has a special combinatorial property
called \emph{complete replicability}.
The notion of replicability can be defined for a formal $q$-series
(i.e.~we do not necessarily assume its convergence and modular invariance).\\

Let 
\[
f(q) = q^{-1} + \sum_{m=1}^\infty a_m q^m \;\;(a_m\in \mathbb{C})
\]
be a formal $q$-series.
For each integer $n\geq 1$, there exists a unique polynomial $P_{f,n}(X)\in\mathbb{C}[X]$
satisfying
\[
P_{n,f}(f(q)) = q^{-n} + O(q) \;\;\; \text{as}\;\;\; q \rightarrow 0.
\]
The polynomials $P_{n,f}(X)$ is called the \emph{Faber polynomial} for $f$ of order $n$.
Each $P_{n,f}$ is a monic polynomial of degree $n$, and it can be described in terms of the coefficients of the formal $q$-series $f(q)$.
For example, we have
\[
P_{1,f}(X)=X,~ P_{2,f}(X)=X^2-2a_1,~ P_{3,f}(X)=X^3-3a_1X-3a_2.
\]
For the formal series $f(q)$ and its Faber polynomials, we define
a double sequence $\{h_{m,n}\}$ by
\[
P_{n,f}(f(q)) = q^{-n} + n \sum_{m=1}^\infty h_{m,n} q^m,
\]
and it can be seen that the double sequence ${h_{m,n}}$ is symmetric.
We now define the notions of replicability for formal series.
\begin{definition}[replicability]
  A formal $q$-series $f(q) = q^{-1} + \sum_{m=1}^\infty a_m q^m$ is said to be \emph{replicable} if the equality $h_{m,n}=h_{r,s}$ holds whenever $\text{gcd}(m,n)=\text{gcd}(r,s)$ and $\text{lcm}(m,n)=\text{lcm}(r,s)$.
\end{definition}

It is known that this defining condition can be written in terms of Hecke-like operator (cf.~\cite{replicable introduction}) :

\begin{proposition}
  A formal series $f(q)$ is replicable if and only if the following conditions are satisfied :\\
  For all $n\geq1$ and its positive divisors $a$ of $n$, 
  there exists a formal series $f^{(a)}$ of the form
  \[
  f^{(a)}(q) = q^{-1} + O(q) \;\;\; \text{as}\;\;\; q \rightarrow 0
  \]
  such that the equalities
  \[
  P_{n,f}(f(q)) = \sum_{\substack{ad=n\\ 0\leq b<d}} f^{(a)} (\zeta_d^b q^{a/d})
  \]
  hold. 
\end{proposition}
We call these $q$-series $f^{(a)}$, $a$-th \emph{replicates} of $f$,
  and set $f^{(1)}=f$.

\begin{remark}
 It is known that if $f(q) =q^{-1}+ \sum_{m=1}^\infty a_mq^m$ is replicable,
 replicates of $f$ are given explicitly in terms of the double sequence $\{h_{m,n}\}$ and the M\"{o}bius function (cf.~\cite{replicable introduction}). 
\end{remark}

We now consider a stronger property, called \emph{complete replicability}.

\begin{definition}[complete replicability]
  Let $f(q) = q^{-1} + \sum_{m=1}^\infty a_m q^m $ be a formal $q$-series.
  We say $f(q)$ is \emph{completely replicable},
  if for any integers $s,t\geq 1$, the $s$-th replicate $f^{(s)}$ is replicable, and the equalities 
  \[
  \bigl(f^{(s)}\bigr)^{(t)}=f^{(st)}
  \]
  are satisfied.
\end{definition}

\begin{definition}
  Suppose that a $q$-series $f(q) = q^{-1} + \sum_{m=1}^\infty a_m q^m $ is completely replicable. Let $k$ be a positive integer.
  We say $f(q)$ is \emph{completely replicable of order $k$} if for any integer $s\geq 1$, the equality 
  \[
  f^{(s,k)}=f^{(s)}
  \]
  are satisfied.
  If $f(q)$ is of order $k$ for some $k\geq1$, $f(q)$ is said to be \emph{completely replicable of finite order}.
\end{definition}

\begin{example}
  The $J$-function ($=j-744$) is a completely replicable function of order 1.
  We can see this in the following way. Consider the sum
  \[
  s_n:= \sum_{\substack{ad=n\\ 0\leq b<d}} J\Bigl(\frac{az+b}{d}\Bigr),
  \]
  and it turns out to be a modular function for $SL_2(\mathbb{Z})$,
  hence $s_n$ is a polynomial of $J$.
  In addition, $s_n$ has a $q$-expansion of the form $q^{-n}+O(q)$.
  Therefore, we get
  \[
  P_{n,J}\bigl(J(q)\bigr) = s_n,
  \]
  hence $J$ is replicable, and all the replicates are $J$ itself.
  It is shown (cf.~\cite{Kozlov} Theorem 3.11.) that
  the $J$-function is the only completely replicable function of order 1
  which has non-trivial group of symmetries.
  
\end{example}

\begin{example}
  Each of the McKay-Thompson series $T_g$ of monstrous moonshine is a completely replicable function whose order coincides with the order of $g$ as an element of the finite Monster group. 
  Borcherds \cite{Borcherds} proved this fact by using the theory of infinite dimensional Lie algebra,
  which was a key point of his proof of monstrous moonshine.
\end{example}

\subsection{Replicability and Hauptmodul property}

As we have seen in the previous section, all the McKay-Thompson series are completely replicable functions, while they are Hauptmoduln for some genus-zero group.
So, what is the relationship between replicability and Hauptmodul property?
Conway-Norton \cite{Rational Hauptmodul} proved that Hauptmodul with rational $q$-coefficients for moonshine-type group is replicable.
Conversely, Norton conjectured the following~:
\begin{conjecture}[Norton]
 Suppose a formal $q$-series $f = q^{-1}+\sum_{n=1}^\infty a_nq^n$ is a replicaple, and that it has rational coefficients.
 Then the $q$-series $f$ is either a Hauptmodul a genus-zero group of moonshine-type,
 or $q^{-1}, q^{-1}\pm q$. 
\end{conjecture}
Therefore, if we assume the Norton's conjecture,
we would get a corollary of Theorem \ref{main thm} that CM-values of replicable functions with rational integer coefficients are algebraic integers, eliminating the degenerate cases.
However, if we restrict ourselves to consider completely replicable $q$-series of finite order, it turned out that we can avoid assuming the above conjecture, by combining the results of \cite{Cummins-Gannon} and \cite{Kozlov}.

\subsection{Relation to modular equations}
Kozlov \cite{Kozlov} studied the relation between complete replicability and modular equations, and proved that completely replicable functions satisfies many modular equations.
To be precise, Kozlov proved the following~:
\begin{proposition}[\cite{Kozlov}~Proposition 4.1.]
  Let $f(q)$ be a completely replicable function of order $k$.
  Then $f(q)$ satisfies a modular equation of all order $n$ coprimes to $k$.
\end{proposition}
By using Theorem 1.3 of \cite{Cummins-Gannon}, which claims that if a $q$-series satisfies a modular equation of any order $n\equiv 1 \mod k$, then $f(q)$ is either a Hauptmodul for a moonshine-type group or some degenerate type, 
we can get the following~:
\begin{proposition}
  Let $f(q)= q^{-1}+ \sum_{m=1}^\infty a_m q^m$ be a completely replicable function of order $k$,
  and suppose that all the coefficients are algebraic integers.
  Then, $f(q)$ is either a Hauptmodul of some moonshine-type group which contains congruence subgroup $\Gamma_0(N)$, or a degenerate type $q^{-1}+\zeta q$ where
  $\zeta^{\text{gcd}(24,k)}=1$.
\end{proposition}
Consequently, we get the following result as a corollary of Theorem.\ref{main thm} of the present paper :
\begin{corollary}\label{completely replicable}
 Let $f(q)$ be a completely replicable function of finite order, with cyclotomic integer $q$-coefficients.
 We assume that $f(q) \neq q^{-1}+ a_1q$.
 Then, there exists an integer $N\geq1$ such that
if $\tau\in\mathbb{H}$ be a CM point satisfying $a\tau^2+b\tau+c=0~~(a,b,c\in\mathbb{Z}, (a,N)=1)$,
and
 the CM value $f(\tau)$ is an algebraic integer.
\end{corollary}

\bibliographystyle{amsplain}

\begin{thebibliography}{10}


\bibitem{Schwarz}
A.~Basraoui and J.~McKay,
{\em The Schwarzian equation for completely replicable functions.}
\newblock LMS Journal of Computation and Mathematics, 20(1):30-52, 2017. 



\bibitem{Borcherds}
R.~Borcherds,
{\em Monstrous moonshine and monstrous Lie superalgebras.}
\newblock Invent.Math.109,405-44, 1992. 

\bibitem{mock modular}
K.~Bringmann,~A.~Folsom,~K.~Ono, and L.~Rolen,
{\em Harmonic Maass Forms and Mock Modular Forms: Theory and Applications.}
\newblock A.M.S.Colloquium Publications,vol.64, 2017. 

\bibitem{Carnahan 1}
S.~Carnahan,
{\em Generalized Moonshine I: Genus zero functions.}
\newblock Algebra and
Number Theory 4, no.6, 649–679, 2010. 

\bibitem{Carnahan 2}
S.~Carnahan,
{\em Generalized Moonshine II: Borcherds products.}
\newblock Duke Math J.
161, no.5, 893–950, 2012.


\bibitem{Chen-Yui}
I.~Chen and N.~Yui,
\newblock {\em Singular values of Thompson series.}
\newblock Groups,Difference Sets,and the Monster, 255-326, 1993. 






\bibitem{ATLAS}
J.H.~Conway, R.T.~Curtis, S.P.~Norton, R.A.~Parker and R.A.~Wilson,
\newblock {\em An Atlas of Finite Groups.}
\newblock Oxford University Press, 1985. 



\bibitem{Conway-Norton}
J.H.~Conway and S.P.~Norton ,
\newblock {\em Monstrous Moonshine.}
\newblock Bull.London.Math.Soc.11,308-39, 1979. 

\bibitem{Rational Hauptmodul}
J.H.~Conway and S.P.~Norton ,
\newblock {\em Rational Hauptmoduls are replicable.}
\newblock Can. J. Math.,47, 1201-1218, 1995.




\bibitem{Cox}
D.~Cox,
\newblock{\em Primes of the forms $x^2+ny^2$, 2nd ed}.
\newblock Pure and Applied Mathematics(Hoboken),John Wiley\&Sons,Inc., 2013.

\bibitem{principal moduli}
  D.~Cox, J.~McKay, and P.~Stevenhagen,
\newblock {\em Principal Moduli and Class Fields.}
\newblock Bulletin of the London Mathematical Society,vol.36,2003.



\bibitem{Cummins}
C.~Cummins,
\newblock {\em Congruence Subgroups of Groups Commensurable
 with PSL$(2,\mathbb{Z})$ of Genus 0 and 1}.
\newblock Experiment.Math. 13,361-82, 2004.

\bibitem{Cummins-Gannon}
C.~Cummins and T.~Gannon,\;{\em Modular equations and the genus zero property of moonshine functions.} \;{Invent.Math.129,413-43, 1997}.




\bibitem{Moonshine}
J.~Duncan, M.~Griffin and K.~Ono, {\em Moonshine.} 
Res. Math. Sci. 2, 11, 2015.


\bibitem{More on}
D.~Ford, J.~McKay and S.~Norton, \;{\em More on replicable functions.} 
Communications in Algebra, 22(13), 5175–5193, 1994.


\bibitem{F-L-M}
I.~Frenkel,~J.~Lepowsky and A.~Meurman,
\newblock{\em Vertex Operator Algebras and the Monster.}
\newblock San Diego Academic Press, 1988.




\bibitem{Gannon_Moonshine}
T.~Gannon,
\newblock{\em Moonshine beyond the Monster-The Bridge Connecting Algebra,Modular Forms and Physics.}
\newblock Cambridge University Press, 2006.


\bibitem{Gannon_Hauptmodul}
T.~Gannon,
\newblock {\em The Algebraic Meaning of Being a Hauptmodul.}
\newblock  London Mathematical Society Lecture Note Series. Cambridge University Press,204-218, 2010.

\bibitem{multiplicative symmmetry}
B.~Heim and A.~Murase,
\newblock {\em Completely replicable functions and symmetries.}
\newblock  Abhandlungen aus dem Mathematischen Seminar der Universität Hamburg 89(3), 2019.


\bibitem{Jorgenson}
J.~Jorgenson, L.~Smajlović, H.~Then,
\newblock {\em The Hauptmodul at elliptic points of certain arithmetic groups,}
\newblock Journal of Number Theory,Volume 204, 661-682, 2019.




\bibitem{Koo Dong Shin}
J.K.~Koo, D.H.~Shin,
\newblock {\em Singular values of principal moduli,}
\newblock Journal of Number Theory, Volume 133, Issue 2, 475-483, 2013.



\bibitem{Kozlov}
D.~Kozlov,
\newblock {\em On completely replicable functions and external poset theory.}
\newblock MSc thesis,University of Lund, Sweden, 1994.

\bibitem{Lang alg}
S.~Lang,
\newblock {\em Algebraic number theory}
\newblock GTM 110, New York,Springer, 1986.

\bibitem{Lang elliptic}
S.~Lang,
\newblock {\em Elliptic Functions,2nd edn.}
\newblock  New York,Springer, 1997.


\bibitem{replicable introduction}
J.~McKay and A.~Sebbar,
\newblock {\em Replicable Functions: An Introduction.}
\newblock  In: P.~Cartier, P.~Moussa, B.~Julia, P.~Vanhove. (eds) Frontiers in Number Theory, Physics, and Geometry II. Springer, Berlin, Heidelberg, 2007.






\bibitem{Shimura}
G.~Shimura,
\newblock {\em Introduction to the Arithmetic Theory of Automorphic Functions.}
\newblock Iwanami Shoten and Princeton Univ. Press, 1971.




\bibitem{Zhu}
Y.~Zhu, 
\newblock {\em Modular invariance of characters of vertex operator algebras.}
\newblock J. Amer. Math. Soc. 9(1), 237-302. 1996.






\end{thebibliography}

\end{document}